\documentclass{amsart}
\usepackage{amsfonts}
\usepackage{amsmath,amscd}
\usepackage{amssymb}
\usepackage{amsthm}
\usepackage{newlfont}
\newcommand{\f}{\frac}

\newcommand{\ds}{\displaystyle}

 \newtheorem{thm}{Theorem}[section]
 
 \newtheorem{lem}[thm]{Lemma}
 
 \theoremstyle{definition}
 
 \theoremstyle{remark}

 \numberwithin{equation}{section}

\begin{document}

\title[Characterizing finite $p$-groups by their Schur multipliers, $t(G)=5$]
 {Characterizing finite $p$-groups by their Schur multipliers, $t(G)=5$}

\author[P. Niroomand]{Peyman Niroomand}
\address{School of Mathematics and Computer Science\\
Damghan University of Basic Sciences, Damghan, Iran}
\email{p$\_$niroomand@yahoo.com}






\begin{abstract}
Let $G$ be a finite $p$-group of order $p^n$. It is known that
$|\mathcal{M}(G)|=p^{\f{1}{2}n(n-1)-t(G)}$ and  $t(G)\geq 0$. The
structure of $G$ characterized when $t(G)\leq 4$ in
\cite{be,el,ni,sa,zh}. The structure description of $G$ is
determined in this paper for $t(G)=5$.
\end{abstract}

\maketitle
\section{introduction}
Let $G$ be a finite $p$-group and $\mathcal{M}(G)$ denotes the Schur
multiplier of $G$. It is known that
$|\mathcal{M}(G)|=p^{\f{1}{2}n(n-1)-t(G)}$, where $t(G)\geq 0$ by
the result of Green in \cite{ge}.

The Structure of $G$ is determined when $t(G)=0,1$ in \cite{be}. In
the case $t(G)=2$ and $3$, Zhou in \cite{zh} and Ellis in \cite{el}
determined the structure of $G$, respectively.  Recently all finite
$p$-group $G$ when $t(G)=4$ are listed in \cite{ni} by the author.
In the present paper, structure of all finite non-abelian $p$-groups
will be given when $t(G)=5$. Our method is quite different to that
of \cite{be,el,zh} and depends on the results of \cite{ni1,ni2}.
\section{Notations and preparatory results}
We use notations and terminology of \cite{el,ni}. In this paper,
$D_8$ and $Q_8$ denote the dihedral and quaternion group of order
$8$, $E_1$ and $E_2$ denote the extra special $p$-groups of order
$p^3$ of exponent $p$ and $p^2$, respectively.  $E_4$ denotes the
unique central product of a cyclic group of order $p^2$ and a
non-abelian group of order $p^3$. Also ${\mathbb{Z}}^{(m)}_{p^{n}}$
denotes the direct product of $m$ copies of the cyclic group of
order $p^n$.  We say that $G$ has the property $t(G)=5$ or briefly
with $t(G)=5$ if the order its Schur multiplier is equal to
$p^{\f{1}{2}n(n-1)-5}$.

 We state some
essential theorems which play important roles in the proof of our
Main Theorem, without proof as follows.
\begin{thm}$\mathrm{(See}$ \cite[Main Theorem]{ni1}$\mathrm{).}$\label{2} Let $G$ be a non-abelian finite
$p$-group of order $p^n$. If $|G^{'}| = p^k$, then we have
\[|\mathcal{M}(G)|\leq p^{\frac{1}{2}(n+k-2)(n-k-1)+1}.\] In
particular,
\[|\mathcal{M}(G)|\leq p^{\frac{1}{2}(n-1)(n-2)+1},\]
and the equality holds  in this last bound if and only if  $G=
E_1\times Z$, where $Z$ is an elementary abelian $p$-group.
\end{thm}

\begin{thm}$\mathrm{(See}$ \cite[Theorem 2.2.10]{kar}\label{3}$\mathrm{).}$ For every finite groups $H$ and $K$, we have
\[\mathcal{M}(H\times K)\cong\mathcal{M}(H)\times \mathcal{M}(K)\times \ds\frac{H}{H'}\otimes\ds\frac{K}{K'}.\]
\end{thm}

\begin{thm}$\mathrm{(See}$ \cite[Theorem 3.3.6]{kar}\label{4}$\mathrm{).}$
Let G be an extra special $p$-group of order $p^{2m+1}$.
Then
\begin{itemize}
\item[(i)]  If $m\geq 2$, then ${|\mathcal M}(G)|=p^{2m^2-m-1}$.
\item[(ii)] If $m=1$, then the order of Schur multipliers of $D_8, Q_8, E_1$ and $E_2$ are
equal to
 $2,1,p^2$ and $1$, respectively.
\end{itemize}
\end{thm}
\section{Main Theorem}
In this section we intend to characterize all finite non-abelian
$p$-groups with the property $t(G)=5$. In fact, we have
\begin{thm}[Main Theorem]Let $G$ be a non-abelian $p$-group
of order $p^n$. Then \[|\mathcal{M}(G)|=p^{\f{1}{2}n(n-1)-5}\] if
and only if $G$ is isomorphic to one of the following groups.
\begin{itemize}
\item[(1)]$D_8\times {\mathbb{Z}}^{(3)}_2$,
\item[(2)]$E_1\times {\mathbb{Z}}^{(4)}_p$,
\item[(3)]$E_2\times {\mathbb{Z}}^{(2)}_p$,
\item[(4)]$ E_4\times{\mathbb{Z}}_p$,
\item[(5)]extra special $p$-group of order $p^5$,
\item[(6)]$\langle a,b~|~a^{p^2}=1,
b^{p^2}=1,[a,b,a]=[a,b,b]=1, [a,b]=a^p\rangle$,
\item[(7)]$\langle a,b~|~a^{p^2}=b^p=1,[a,b,a]=[a,b,b]=a^p,[a,b,b,b]=1
\rangle$,
\item[(8)]$\langle
a,b~|~a^{p^2}=b^p=1,[a,b,a]=1,[a,b,b]=a^{np},[a,b,b,b]=1 \rangle,$
where n is a fixed quadratic non-residue of $p$ and $p\neq 3$,
\item[(9)]$\langle
a,b~|~a^{p^2}=1,b^3=a^3,[a,b,a]=1,[a,b,b]=a^{6},[a,b,b,b]=1
\rangle,$
\item[(10)]$\langle
a,b~|~a^p=1,b^p=[a,b,b],[a,b,a]=[a,b,b,a]=[a,b,b,b]=1 \rangle$,
\item[(11)]$D_{16}$,
\item[(12)]$~\langle a,b~|~a^4=b^4=1,a^{-1}ba=b^{-1} \rangle$,
\item[(13)]$Q_{8}\times {{\mathbb{Z}}}^{(2)}_{2}$,
\item[(14)]$(D_8\times{\mathbb{Z}}_{2})\rtimes{\mathbb{Z}}_{2}$,
\item[(15)]$(Q_8\times{\mathbb{Z}}_{2})\rtimes{\mathbb{Z}}_{2}$,
\item[(16)]${\mathbb{Z}}_{2}\times \langle
a,b,c~|~a^2=b^2=c^2=1, abc=bca=cab\rangle.$

\end{itemize}
\end{thm}
We separate the proof of it into several steps as follows.
\begin{lem}\label{m1} Let $G$ be a $p$-group of order $p^n$ and  $|G^{'}|= p^k (k\geq 2)$ with $t(G)=5$.
Then $n\leq 4$ unless $k=2$, in this case $n\leq 6.$
\end{lem}
\begin{proof}By virtue of Theorem \ref{2}, we have
\[\f{1}{2}(n^2-n-10)\leq \frac{1}{2}(n+k-2)(n-k-1)+1\leq
\f{1}{2}n(n-3)+1,\] which follows the result.
\end{proof}

\begin{thm}\label{m2} Let $G$ be a non-abelian finite $p$-group of order $p^n$
with $t(G)=5$.  Then $|G|\leq p^7$. In the case that $n=6$ and
$n=7$, $G$ is isomorphic to \[D_8\times {\mathbb{Z}}^{(3)}_2
~\text{and}~E_1\times {\mathbb{Z}}^{(4)}_p,\] respectively.
\end{thm}
\begin{proof} One can easily check that $n\leq
7$ by Theorem \ref{2}.

In the case $n=7$, Lemma \ref{m1} follows that $|G^{'}|=p$. Since
$|\mathcal{M}(G)|=p^{16}$ and equality holds in Theorem \ref{2}, we
should have $G\cong E_1\times {\mathbb{Z}}^{(4)}_p$. When $n=6$,
$|\mathcal{M}(G)|=p^{10}$ and by a consequence of \cite[Main
Theorem]{ni2}, we have $G\cong D_8\times {\mathbb{Z}}^{(3)}_2$.
\end{proof}
As mentioned in the Lemma \ref{m1} and Theorem \ref{m2}, we may
assume that $n\leq 5$.  First assume that $p\neq 2$.

\begin{thm}\label{tm1} Let $|G|=p^5$ $(p\neq 2)$ and $|G^{'}|\geq
p^2$. Then there is no group with $t(G)=5$.
\end{thm}
\begin{proof} Using Lemma \ref{m1}, we may assume that $|G^{'}|=p^2$.

For each central subgroup $K$ of order $p$, \cite[Theorem 4.1]{kar}
implies that
\[p^{5}=|\mathcal{M}(G)|\leq p^2~|\mathcal{M}(G/K)|.\]
If for every central subgroup $K$, $|\mathcal{M}(G/K)|=p^4$ the
proof of \cite[Main Theorem]{ni2} shows that $G\cong
{\mathbb{Z}}^{(4)}_p\rtimes {\mathbb{Z}}_p$ and hence
$|\mathcal{M}(G)|=p^6$, which is a contradiction. Thus there exists
a central subgroup $K$ such that $|\mathcal{M}(G/K)|\leq p^3$. Since
$p\neq 2$ and $|G/K|=p^4$, \cite[Main Theorem]{ni2} follows that
$|\mathcal{M}(G/K)|\leq p^2$, and so $|\mathcal{M}(G)|\leq p^4$,
which contradicts the assumption.
\end{proof}
\begin{thm}\label{tm} Let $|G|=p^5$ $(p\neq 2)$ and  $|Z(G)|=p^3$ with
$t(G)=5$, then $G$ is isomorphic to \[E_2\times {\mathbb{Z}}^{(2)}_p
~\text{or}~ E_4\times{\mathbb{Z}}_p.\]
\end{thm}
\begin{proof}
 It is known by \cite[Theorem 4.1]{kar} that,
\[|\mathcal{M}(G)||G^{'}|\leq |\mathcal{M}(G/G^{'})||G^{'}\otimes
G/Z(G)|.\] We know that  $|G^{'}|=p$ by Theorem \ref{tm1}. Now,  if
$G/G^{'}$ is not elementary abelian, then
$|\mathcal{M}(G/G^{'})|\leq p^3$, and so $|\mathcal{M}(G)|\leq p^4$,
which is a Impossible. Therefore, $G/G^{'}$ is elementary abelian.
On the other hand, \cite[Theorem 2.2]{j2} implies that $Z(G)$ is of
exponent at most $p^2$. Thus two cases may be considered.

Case $I$. First suppose that $Z(G)$ is of exponent $p$. By a result
of \cite[Lemma 2.1]{ni1}, we should have $G\cong H\times
{\mathbb{Z}}^{(2)}_p,$ where $H$ is extra special of order $p^3$.
Since $|\mathcal{M}(G)|=p^5$, Theorems \ref{3} and \ref{4} imply
that $H\cong E_2$.

Case $II$. In this case, similar to pervious part one can see that
$G\cong H\times {\mathbb{Z}}_{p^2},$ where $H$ is extra special of
order $p^3$ or $G\cong E_4\times {\mathbb{Z}}_p$. By invoking
Theorems \ref{3} and \ref{4}, the order of the Schur multiplier of
$H\times {\mathbb{Z}}_{p^2}$ is at most $p^4$, and hence does not
have the property $t(G)=5$. On the other hand, by a result of
\cite[Lemma 3.5]{ni} and Theorem \ref{3}, we should have
$|\mathcal{M}(E_4\times{\mathbb{Z}}_p)|=p^5$, as required.
\end{proof}
\begin{thm} Let $|G|=p^5$ $(p\neq 2)$
 and $|Z(G)|=p^2$. Then there is no group with $t(G)=5$.
\end{thm}
\begin{proof} We may assume that $G/G^{'}$ is not elementary abelian by appealing to \cite[Lemma
2.1]{ni1}. Using \cite[Proposition 1]{el2}, $G/Z(G)$ is elementary
abelian and $G/G^{'}\cong
 {{\mathbb{Z}}}^{(2)}_{2}\times \mathbb{Z}_{2} $. Hence $Z(G)$ and
Frattini subgroup coincide, and so \cite[Proposition 1]{el2} (see
also \cite[Proposition 5 (i) and (ii)]{el3}) shows that
\[p^2|\mathcal{M}(G)|\leq |\mathcal{M}(G/G^{'})||G^{'}\otimes G/Z(G)|\leq p^6.\]
Thus $|\mathcal{M}(G)|\leq p^4$, which is a contradiction.
\end{proof}
\begin{lem}Every extra special $p$-group of order $p^5$ has the
property $t(G)=5$.
\begin{proof} It is straightforward by Theorem \ref{4}.
\end{proof}
\end{lem}
\begin{thm} Let $|G|=p^4$ and $|G^{'}|=p$ with $t(G)=5$. Then $G$ is
isomorphic to \[\langle a,b~|~a^{p^2}=1,
b^{p^2}=1,[a,b,a]=[a,b,b]=1, [a,b]=a^p\rangle.\]
\end{thm}
\begin{proof} First suppose that $G/G^{'}$ is elementary. By a
result of \cite[Lemma 2.1]{ni1}, we have  $G\cong H\times
{\mathbb{Z}}_p$ or $G\cong E_4$. The order of Schur multipliers of
both of them is at least $p^2$. Thus $G/G^{'}$ can  not be
elementary abelian. Since $G^p $ and $G^{'}$ are contained in
$Z(G)$, we consider two cases.

Case $I$. First assume that $G^{'}\cap G^p=1$, then $G/G^p\cong
E_1$, and so $|\mathcal{M}(G)|\geq |\mathcal{M}(E_1)|=p^2$ directly
by using \cite[Corollary 2.5.3 (i)]{kar}.

Case $II$. In this case, we have two possibilities for $Z(G)$. The
first possibility is $Z(G)=G^p\cong {\mathbb{Z}}_{p^2}$, thus $G$ is
of exponent $p^3$ and obviously $|\mathcal{M}(G)|=1$.
 The second possibility is
$Z(G)=G^p\cong {\mathbb{Z}}_p\times G^{'}$. By \cite[pp. 87-88]{bu},
there is a unique group of order $p^4$ with this properties, which
is isomorphic to  \[\langle a,b~|~a^{p^2}=1,
b^{p^2}=1,[a,b,a]=[a,b,b]=1, [a,b]=a^p\rangle.\]
\end{proof}
\begin{lem} Let $|G|=p^4$ and $|G^{'}|=p^2$ with $t(G)=5$. Then $G$ is isomorphic
to\[\langle a,b~|~a^{p^2}=b^p=1,[a,b,a]=[a,b,b]=a^p,[a,b,b,b]=1
\rangle,\]
\[\langle a,b~|~a^{p^2}=b^p=1,[a,b,a]=1,[a,b,b]=a^{np},[a,b,b,b]=1
\rangle,\] where n is a fixed quadratic non-residue of $p$ and
$p\neq 3$,
\[\langle a,b~|~a^{p^2}=1,b^3=a^3,[a,b,a]=1,[a,b,b]=a^{6},[a,b,b,b]=1
\rangle,\]
\[\langle a,b~|~a^p=1,b^p=[a,b,b],[a,b,a]=[a,b,b,a]=[a,b,b,b]=1
\rangle.\]
\end{lem}
\begin{proof} The result is obtained from \cite[pp. 4177]{el} and \cite[pp. 88]{bu}
$($see also \cite[pp. 196-198]{sch}$)$.
\end{proof}
\begin{lem} Let $G$ be a $p$-group of order $16$ with $t(G)=5$.
Then $G$ is isomorphic to \[ D_{16}~\text{or}~\langle
a,b~|~a^4=b^4=1,a^{-1}ba=b^{-1} \rangle.\]
\end{lem}
\begin{proof} See table $I$ on \cite{ni4}.
\end{proof}
\begin{lem} Let $G$ be a $p$-group of order $32$ with $t(G)=5$.
Then $G$ is isomorphic to \[ Q_{8}\times
{{\mathbb{Z}}}^{(2)}_{2},(D_8\times{\mathbb{Z}}_{2})\rtimes{\mathbb{Z}}_{2},
(Q_8\times{\mathbb{Z}}_{2})\rtimes{\mathbb{Z}}_{2}
~\text{or}\]\[{\mathbb{Z}}_{2}\times \langle a,b,c~|~a^2=b^2=c^2=1,
abc=bca=cab\rangle.\]
\end{lem}
\begin{proof} These groups are obtained by using the HAP package \cite{el5}  of GAP
\cite{gap}.
\end{proof}

\end{document}